\documentclass[11pt]{article}
\usepackage{amssymb,latexsym,amsmath,enumerate,verbatim,amsfonts,amsthm}
\usepackage{cite}

\usepackage{hyperref}
\numberwithin{equation}{section}
\newtheorem{definition}{Definition}[section]
\newtheorem{theorem}{Theorem}[section]
\newtheorem{corollary}[theorem]{Corollary}

\newtheorem{lemma}{Lemma}[section]

\title{\bf The analytic criterion of strict copositivity for a 4th-order 3-dimensional tensor\thanks{This work was supported by the National Natural Science Foundation of P.R. China (Grant No.12171064), by The team project of innovation leading talent in chongqing (No.CQYC20210309536) and by the Foundation of Chongqing Normal university (20XLB009).}}
\author{Mingjung Sheng, Yisheng Song\thanks{Corresponding author: yisheng.song@cqnu.edu.cn}\\
	{\normalsize School of Mathematical Sciences, Chongqing Normal University} \\
	{\normalsize Chongqing, 401331, P.R. China.}\\
	{\normalsize mingjun\_2001@163.com(Sheng); yisheng.song@cqnu.edu.cn (Song)}
}

\begin{document}
\date{ }
\maketitle

\begin{abstract}
This paper focuses on the strict copositivity analysis of 4th-order 3-dimensional symmetric tensors.
A necessary and sufficient condition is  provided for the strict copositivity of a fourth-order symmetric tensor.
Subsequently, building upon this conclusion, we  discuss the strict copositivity of  fourth-order three-dimensional symmetric tensors with its entries  $\pm 1, 0$, and further build  their necessary and sufficient conditions. Utilizing these theorems, we can effectively verify the strict copositivity of a general fourth-order three-dimensional symmetric tensors.


\end{abstract}

\section{Introduction}

Tensors represent a significant concept in mathematics, serving as a generalization of vectors and matrices. Recently, the copositivity of tensors has garnered considerable attention due to its importance in polynomial optimization \cite{qi2005eigenvalues,qi2013symmetric,nie2018complete,qi2018tensor,song2015necessary}, hypergraph theory\cite{chen2018copositive,nie2018complete,wang2024copositive}, complementarity problems\cite{huang2019tensor1,qi2019tensor,huang2019tensor2,song2016properties,song2017strictly,song2018structural}, and particle physics\cite{kannike2012vacuum,kannike2016vacuum,chen2018copositive,song2022copositivity,song2023co,song2023vacuum}, among others. A notable application is the evaluation of vacuum stability in scalar dark matter models\cite{belanger2012impact,belanger2014minimal,kannike2016vacuum,kannike2012vacuum}, which can be assessed through the co-positivity of the corresponding tensor. Kannike\cite{ivanov2020yet} demonstrated that the copositivity of tensors serves as a sufficient condition for the boundedness from below of scalar potentials, thereby laying the groundwork for subsequent research, including the analysis of vacuum stability in $\mathbb{Z}_3$ scalar dark matter models \cite{song2022copositivity}. Thus, the development of copositive tensor theory has provided valuable insights into the vacuum stability of scalar dark matter models\cite{kannike2016vacuum,ivanov2012discrete,ivanov2013classification}.

The study of copositive matrices dates back to Motzkin's work in 1952\cite{motzkin1952copositive}, and Baumert\cite{baumert1966extreme} explored extremal copositive quadratic forms. Cottle et al.\cite{cottle1970classes} contributed to the foundational knowledge by classifying copositive matrices. Subsequent researchers such as Simpson-Spector\cite{simpson1982copositive}, Hadeler\cite{hadeler1983copositive}, Nadler\cite{nadler1992nonnegativity}, Chang-Sederberg\cite{chang1994nonnegative}, and Andersson Chang-Elfving\cite{andersson1995criteria} have elucidated the (strict) copositivity conditions for $2 \times 2$ and $3 \times 3$ matrices, providing essential support for the study of higher-order tensor copositivity. In 2013, Qi\cite{qi2013symmetric} introduced the concept of copositive tensors, extending the notion of copositive matrices, establishing their fundamental properties, and indicating that symmetric non-negative tensors and semi-positive definite tensors are copositive. Song-Qi\cite{song2015necessary} made a significant contribution in 2015 by proposing necessary and sufficient conditions for tensor copositivity, proving that the necessary and sufficient condition for a symmetric tensor to be (strictly) copositive is that none of its principal sub-tensors possess (non-positive) negative eigenvalues. In 2016, Song-Qi\cite{song2016eigenvalue} introduced the concepts of Pareto H-eigenvalues and Pareto Z-eigenvalues, linking these concepts with tensor copositivity. Song-Qi\cite{song2021analytical} also associated tensor complementarity problems with copositive tensors, facilitating the development of methods for solving complementarity problems arising in particle physics. Qi-Chen-Chen\cite{qi2018tensor} further advanced the theory of tensor eigenvalues and its applications in 2018, offering a comprehensive framework for analyzing copositive tensors.

Recently, Liu-Song\cite{liu2022copositivity} derived sufficient conditions for the copositivity of third-order symmetric tensors and demonstrated their applicability in $\mathbb{Z}_3$ scalar dark matter. Building on this, Song-Li\cite{song2022copositivity} presented necessary and sufficient conditions for the copositivity of fourth-order symmetric tensors, contributing to the verification of vacuum stability in Higgs scalar potential models, a critical aspect of particle physics. Song-Liu\cite{song4866847copositivity}  proposed analytical necessary and sufficient conditions for the (strict) copositivity of fourth-order three-dimensional symmetric tensors with entries of $1$ or $-1$, enabling the validation of the copositivity of a general fourth-order three-dimensional tensor. However, an explicit expression for the copositivity of higher-order tensors remains elusive.

In this paper, inspired by the works of Hoffman, Alan J., and Francisco Pereira \cite{hoffman1973copositiv}, Liu-Song\cite{liu2022copositivity} , Song-Li\cite{song2022copositivity}, Song-Liu\cite{song4866847copositivity}, and related studies, it is straightforward to obtain a necessary and sufficient conditions for the strict copositivity of fourth-order two-dimensional symmetric tensors. We propose a sufficient and necessary condition for the strict copositivity of a fourth-order symmetric tensor, followed by a specific case involving fourth-order three-dimensional symmetric tensors with entries of $1$ or $-1$, refining the theory established in \cite{song4866847copositivity}. Finally, we discuss the strict copositivity of special fourth-order three-dimensional symmetric tensors with entries of $-1, 0$, or $1$, aiming to provide a more comprehensive understanding of tensor copositivity.

\section{Preliminaries and Basic Facts}

\begin{definition}
 An $m$th-order $n$-dimensional symmetric tensor $\mathcal{T}= (t_{i_1 i_2 ...i_m})$ is called
\begin{itemize}
	\item[(i)] \textbf{copositive}  \cite{qi2013symmetric} if
	$\mathcal{T} x^m = \sum\limits_{i_1,i_2,\dots ,i_m=1}^{n} t_{i_1 i_2...i_m} x_{i_1}x_{i_2}...x_{i_m} \geq 0$ for all nonegative vector $x=(x_1,x_2,...,x_n)^T$;
	\item[(ii)]   \textbf{strictly  copositive} \cite{qi2013symmetric} if $\mathcal{T} x^m >0$ for all nonegative and nonzero vector $x=(x_1,x_2,...,x_n)^T$;
	\item[(iii)]  \textbf{positive   (semi)-definite} \cite{qi2005eigenvalues} if $\mathcal{T} x^m \geq(>) 0$ for all nonzero vector $x \in \mathbb{R}^n$ and an even positive integer $m$.
\end{itemize}
\end{definition}
\begin{lemma}\label{lem:21} \cite{qi2013symmetric} Suppose an $m$th-order $n$-dimensional symmetric tensor	$\mathcal{T}= (t_{i_1 i_2 ...i_m})$ is copositive. If $t_{ii\cdots i}=0$, then $t_{ii\cdots ij}\ge0$ for all $j.$
\end{lemma}
The (strictly) copositive conditions  of $2\times2$ symmetric matrices were    showed by Andersson-Chang-Elfving \cite{andersson1995criteria}, Chang-Sederberg \cite{chang1994nonnegative}, Hadeler \cite{hadeler1983copositive} and Nadler \cite{nadler1992nonnegativity}, Simpson-Spector \cite{simpson1982copositive}.
\begin{lemma}\label{lem:22} Let $M=(m_{ij})$ be a symmetric $2\times2$ matrix. Then $M$ is (strictly) copositive if and only if
	$$
		m_{11} \geq 0 (>0), m_{22} \geq 0 (>0),\alpha=m_{12}+\sqrt{m_{11}m_{22}}\geq 0 (>0).$$
\end{lemma}
Schmidt-He$\beta$\cite{schmidt1988positivity}, Ulrich-Watson \cite{ulrich1994positivity} and Qi-Song-Zhang \cite{2022POSITIVITY} provided the analytic conditions for the nonnegativity of a quartic (cubic) and univariate polynomial in \(\mathbb{R}^+\). By applying these results, the copositive conditions of a 4th-order (3rd-order) 2-dimensional tensor were easily proved. Also see Song-Li \cite{song2022copositivity} and Liu-Song \cite{liu2022copositivity} for more details.
\begin{lemma}\label{lem:23}
Let $\mathcal{T} = (t_{ijkl})$ is a 4th-order 2-dimensional symmetric tensor with $t_{1111}>0$ and $t_{2222}>0$, then $\mathcal{T}$ is copositive if and only if
$$\begin{cases}
\Delta \leq 0, t_{1222} \sqrt{t_{1111}} + t_{1112} \sqrt{t_{2222}}>0; \\
t_{1222} \geq 0, t_{1112}\ge0, 3t_{1122}+\sqrt{t_{1111} t_{2222}} \geq0;\\
\Delta \geq 0,\\
|t_{1112} \sqrt{t_{2222}} - t_{1222} \sqrt{t_{1111}} |\leq \sqrt{6t_{1111} t_{1122} t_{2222} + 2t_{1111} t_{2222} \sqrt{t_{1111} t_{2222}}}\\
(i) - \sqrt{t_{1111} t_{2222}} \leq 3t_{1122} \leq 3\sqrt{t_{1111}t_{2222}};\\
(ii) t_{1122}>\sqrt{t_{1111}t_{2222}}~and\\
t_{1112} \sqrt{t_{2222}} + t_{1222} \sqrt{t_{1111}} \geq -\sqrt{6t_{1111}t_{1122}t_{2222}- 2t_{1111}t_{2222} \sqrt{t_{1111}t_{2222}}},
\end{cases} $$
where $\Delta = 4 \times 12^3 (t_{1111}t_{2222}-4t_{1112}t_{1222}+3t_{1122}^2)^3- 72^2 \times 6^2 (t_{1111}t_{1122} t_{2222} + 2t_{1112} t_{1122}t_{1222} - t_{1122}^3 -t_{1112}^2 t_{2222} -t_{1111} t_{1222}^2)^2$.
\end{lemma}

\begin{lemma} \label{lem:24} A 3rd order 2-demensional tensor $\mathcal{T} = (t_{ijk})$ is copositive if and only if $t_{111} \ge 0$, $t_{222} \ge 0$ and $$\begin{cases}
		t_{112} \ge 0, t_{122} \ge 0;\\
	\max\{t_{111}, t_{222}\}>0 \mbox{ and }	4t_{111}t_{122}^3 + 4 t_{112}^3t_{222} + t_{111}^2t_{222}^2 - 6t_{111}t_{112}t_{122}t_{222} - 3t_{112}^2t_{122}^2 \ge 0.
	\end{cases}$$
\end{lemma}

By means of Lemmas \ref{lem:21}, \ref{lem:22}, \ref{lem:23} and \ref{lem:24}, the following lemma may be obtained.
\begin{lemma} \label{lem:25}
Let $\mathcal{T}$ be a 4th order 2-dimensional symmetric tensor with $t_{ijkl}\in\{-1,0,1\}$. Then $\mathcal{T}$ is copositive if and only if
$$\begin{cases}
	t_{1111}=1,  t_{2222}=0,   \begin{cases} t_{1112}\in\{0,1\},  t_{1122}\in\{0,1\}, t_{1222}\in\{0,1\};\\ t_{1222}=0, t_{1122}=-t_{1112}=1;\\   t_{1222}=1, t_{1122}=-1, t_{1112}\in\{0,1\};\end{cases}\\
	t_{1111}=0,  t_{2222}=1, \begin{cases}t_{1112}\in\{0,1\}, t_{1122}\in\{0,1\},  t_{1222}\in\{0,1\};\\t_{1112}=0, t_{1122}=- t_{1222}=1;\\ t_{1112}=1, t_{1122}=-1, t_{1222}\in\{0,1\};\end{cases}\\
	t_{1111}=t_{2222}=0,  \begin{cases} t_{1112}\in\{0,1\}, t_{1222}\in\{0,1\}, t_{1122}\in\{0,1\};\\ t_{1112}=t_{1222}=-t_{1122}=1;\end{cases}\\
	t_{1111}=t_{2222}=1, \begin{cases}	t_{1122}=0, t_{1112}\in\{0,1\}, t_{1222}\in\{0,1\};\\
		t_{1122}=1;\\  t_{1112}=t_{1222}=1.\end{cases}
\end{cases}$$
Moreover, $\mathcal{T}$ is strictly copositive if and only if
$$t_{1111}=t_{2222}=1, \begin{cases}
	t_{1122}=0, t_{1112}\in\{0,1\}, t_{1222}\in\{0,1\};\\
	t_{1112}=t_{1222}=1;\\
	t_{1112}t_{1222}\in\{0,-1\}\mbox{ and }  t_{1122}=1.
\end{cases}$$
\end{lemma}
\begin{proof} Obviously,  the copositivity of $\mathcal{T}$ means $t_{1111}\in\{0,1\}$ and  $t_{2222}\in\{0,1\}$, and then,  it may divides into four different cases.
	
	\textbf{Case 1}. $t_{1111}=0,\ \ t_{2222}=1$, which implies $t_{1112}\ge0$ by Lemma \ref{lem:21}.  That's when $\mathcal{T}x^4$ can be rewritten as
	$$\aligned \mathcal{T}x^4=&4t_{1112}x_1^3x_2+6t_{1122}x_1^2x_2^2+4t_{1222}x_1x^3_2+x_2^4\\
	=&x_2(4t_{1112}x_1^3+6t_{1122}x_1^2x_2+4t_{1222}x_1x_2^2+x_2^3).
	\endaligned$$
	Which is equivalent to $$4t_{1112}x_1^3+3\times2t_{1122}x_1^2x_2+3\times\dfrac43t_{1222}x_1x_2^2+x_2^3\ge0.$$
	From Lemma \ref{lem:24},  it follows that $\mathcal{T}x^4\geq0$ if and only if
	$$ t_{1112}\in\{0,1\},\begin{cases}\mbox{ either }
		t_{1122}\in\{0,1\},\ t_{1222}\in\{0,1\};  \mbox{ or} \\
		4^2t_{1112}^2+4\times 2^3t_{1122}^3 +4^2 \times\left(\dfrac43\right)^3 t_{1222}^3 t_{1112}-3\times2^2\times\left(\dfrac43\right)^2t_{1122}^2 t_{1222}^2\\-6\times \dfrac43\times2\times4t_{1112} t_{1122} t_{1222} \ge0.\end{cases}$$
	If $t_{1112}=0$, then $t_{1122}\in\{0,1\},\ t_{1222}\in\{0,1\};  \mbox{ or} $
	$$t_{1122}^2(t_{1122}-\dfrac23 t_{1222}^2)\ge0 \Leftrightarrow t_{1122}=1, t_{1222}\in\{-1,0,1\}.$$
	If $t_{1112}=1$, then $t_{1122}\in\{0,1\},\ t_{1222}\in\{0,1\};  \mbox{ or} $
	$$27+54t_{1122}^3 +64 t_{1222}^3-36t_{1122}^2 t_{1222}^2-108t_{1122}\ge0 \Leftrightarrow t_{1122}=-1, t_{1222}\in\{0,1\}.$$
	
	\textbf{Case 2}. $t_{1111}=1,\ \ t_{2222}=0$,  the proof is the same as Case 1.
	
	\textbf{Case 3}. $t_{1111}=t_{2222}=0$.  Then for all  $x=(x_1,x_2)^\top\in\mathbb{R}^2_+$, we have
	$$\aligned \mathcal{T}x^4=&4t_{1112}x_1^3x_2+6t_{1122}x_1^2x_2^2+4t_{1222}x_1x_2^3\\
	=&2x_1x_2(2t_{1112}x_1^2+3t_{1122}x_1x_2+2t_{1222}x_2^2)\ge0,
	\endaligned$$
	which is equivalent to $$2t_{1112}x_1^2+3t_{1122}x_1x_2+2t_{1222}x_2^2\ge0.$$
	By Lemma  \ref{lem:22}, $\mathcal{T}x^4\ge0\Leftrightarrow t_{1112}\in\{0,1\},t_{1222}\in\{0,1\}, 3t_{1122}+4\sqrt{t_{1112}t_{1222}}\ge0.$
	That is, $$t_{1112}\in\{0,1\},t_{1222}\in\{0,1\}, t_{1122}\in\{0,1\} \mbox{ or }t_{1112}=t_{1222}=-t_{1122}=1.$$
	
\textbf{Case 4}. $t_{1111}=t_{2222}=1$. 	It follows from Lemma \ref{lem:23} that $\mathcal{T}$ is copositive if and only if
	$$\begin{cases}
		\Delta\leq0\mbox{ and }t_{1112}=t_{1222}=1;\\
		t_{1112}\in\{0,1\},t_{1222}\in\{0,1\}, t_{1122}\in\{0,1\};\\
		\Delta\geq0, t_{1122} \in\{0,1\}\mbox{  and }|t_{1112}-t_{1222}|\leq\sqrt{6t_{1122}+2}.
	\end{cases}$$
	
	Assume $t_{1112}=t_{1222}=1$. Then we have
	$$t_{1122}=1, \Delta=4\times12^3((1-4+3)^{3}-27(1+2-1^3-1-1)^{2})=0,$$
	 or $$t_{1122}=0, \Delta=4\times12^3((1-4+0)^{3}-27(0-0+0-1-1)^{2})<0,$$  or $$t_{1122}=-1, \Delta=4\times12^3((1-4+3)^{3}-27(-1-2+1-1-1)^{2})<0;$$
	So,$$\Delta\leq0\mbox{ and }t_{1112}=t_{1222}=1\ \Leftrightarrow\ t_{1112}=t_{1222}=1.$$
	
	Assume $t_{1122}=1$.   Then   when
	$t_{1112}t_{1222}=1$,  we have $$\Delta=4\times12^3((1-4+3)^{3}-27(1+2-1-1-1)^{2})=0, |t_{1112}-t_{1222}|=0<\sqrt8;$$ or
	when
	$t_{1112}t_{1222}=0$,  we have $$\Delta\ge4\times12^3((1-0+3)^3-27(1+0-1-1-0)^{2})>0, |t_{1112}-t_{1222}|\le1<\sqrt8;$$ or when $t_{1112}t_{1222}=-1,$ we have
	$$\Delta=4\times12^3((1+4+3)^{3}-27(1-2-1-1-1)^{2})>0, |t_{1112}-t_{1222}|=2<\sqrt8.$$
	Thus, the conditions, $\Delta\geq0$, $|t_{1112}-t_{1222}|\leq\sqrt{6t_{1122}+2}$ and $t_{1122}=1$ are equivalent to $$t_{1122}=1.$$
	
	Assume $t_{1122}=0$. Then  when  $t_{1112}t_{1222}=1$,  we have $$\Delta=4\times12^3((1-4+0)^{3}-27(0+0-0-1-1)^{2})<0, |t_{1112}-t_{1222}|=0<\sqrt2;$$ or
	when
	$t_{1112}t_{1222}=0$,  i.e., $t_{1112}=0$ or $t_{122}=0$ or $t_{1112}=t_{122}=0$, then  $$\Delta=4\times12^3((1-0+0)^3-27(0+0-0-1-0)^{2})<0, |t_{1112}-t_{1222}|\le1<\sqrt2,$$
or	$$\Delta=4\times12^3((1-0+0)^3-27(0+0-0-0-0)^{2})>0, |0-0|=0<\sqrt2;$$ or when $t_{1112}t_{1222}=-1,$ we have
	$$\Delta=4\times12^3((1+4+0)^{3}-27(0+0-0-1-1)^{2})>0, |t_{1112}-t_{1222}|=2>\sqrt2.$$
	Thus, the conditions, $\Delta\geq0$, $|t_{1112}-t_{1222}|\leq\sqrt{6t_{1122}+2}$ and $t_{1122}=0$ are equivalent to $$t_{1122}=t_{1112}=t_{1222}=0,$$
	which is covered in the second conditions, $t_{1112}\in\{0,1\},t_{1222}\in\{0,1\}, t_{1122}\in\{0,1\}.$
	So the desired conclusions follow.	
	
	Next we show the strict copositivity of  $\mathcal{T}$.  Clearly, $\mathcal{T}$ is copositive,  and then we only need show $$\mathcal{T}x^4=0\mbox{ for }x\in\mathbb{R}^2_+\ \Longrightarrow\ x=0.$$ If $t_{1112}\in\{0,1\},t_{1222}\in\{0,1\}, t_{1122}\in\{0,1\},$
	then the conclusion is obvious. For the remaining cconditions, $\mathcal{T}x^4$ may be rewritten as follows,
	$$\mathcal{T}x^4=\begin{cases}
		x_1^4+4x_1^3x_2-6 x_1^2x_2^2+4x_1x_2^3+x_2^4, \ \ t_{1112}=t_{1222}=1, t_{1122}=-1;\\
		x_1^4+4x_1^3x_2+6 x_1^2x_2^2-4x_1x_2^3+x_2^4, \ t_{1112}=-t_{1222}=1, t_{1122}=1;\\
		x_1^4-4x_1^3x_2+6 x_1^2x_2^2+4x_1x_2^3+x_2^4, \ -t_{1112}=t_{1222}=1, t_{1122}=1;\\
		x_1^4-4x_1^3x_2+6 x_1^2x_2^2+x_2^4, \ t_{1112}=-1, t_{1222}=0, t_{1122}=1;\\
		x_1^4+6 x_1^2x_2^2-4x_1x_2^3+x_2^4, \ t_{1112}=0, t_{1222}=-1, t_{1122}=1.
	\end{cases}$$
	Then solving the equations,  $$0=\mathcal{T}x^4=\begin{cases}
		(x_1^2+x_2^2)^2+4x_1x_2(x_1-x_2)^2;\\
		(x_1-x_2)^4+8x_1^3x_2;\\
		(x_1-x_2)^4+8x_1x_2^3;\\
		(x_1-x_2)^4+4x_1x_2^3;\\
		(x_1-x_2)^4+4x_1^3x_2,
	\end{cases}$$
	we obviously have  $x_1=x_2=0.$
	
	If $t_{1112}=t_{1222}=-1$ and $t_{1122}=1$,  then $$\mathcal{T}x^4=x_1^4-4x_1^3x_2+6 x_1^2x_2^2-4x_1x_2^3+x_2^4=(x_1-x_2)^4,$$
	and so, $\mathcal{T}x^4=0$ when $x_1=x_2>0$.  That's when $\mathcal{T}$ is only copositive, but not strictly copositive.
	This  completes the proof.	
	\end{proof}
	The following conclusion is obvious by Lemma \ref{lem:25}.
\begin{lemma} \label{lem:26}
	Let $\mathcal{T}$ be a 4th order 2-dimensional symmetric tensor with its entries $|t_{ijkl}|=1$. Then $\mathcal{T}$ is strictly copositive if and only if
	$$t_{1111}=t_{2222}=1, \begin{cases}
				t_{1112}=t_{1222}=1;\\
		t_{1112}t_{1222}=-1\mbox{ and }  t_{1122}=1.
	\end{cases}$$
\end{lemma}

\section{Copositivity of 4th-order 3-dimensional symmetric tensors}

\begin{theorem}\label{thm:31}
	 Let $\mathcal{T}=(t_{ijkl})$ be a $4$th-order $n$-dimensional symmetric tensor. Then
 $\mathcal{T}$ is strictly copositive  if and only if $$\begin{cases}
		x\in\mathbb{R}_+^n\mbox{ and  }	\mathcal{T}x^4=0\ \Longrightarrow\ x=0,\\
		\mbox{there is a }y\in\mathbb{R}^n_+\setminus\{0\}\mbox{ such that }\mathcal{T}y^4>0;
	\end{cases}$$
\end{theorem}
\begin{proof}
The necessarity is obvious.  Now we show the sufficiency. Suppose $\mathcal{T}$ is not strictly copositive when the conditions are satisfied. There exists $u\in\mathbb{R}^n_+\setminus \{0\}$ such that $\mathcal{T}u^4\leq0$.  Since $\mathcal{T}u^4=0$ means $u=0$ by the conditions, then $\mathcal{T}u^4<0$. Apply the intermediate value theore to continuous function $\mathcal{T}x^4$, there is an $\lambda\in(0,1)$ such that \begin{center}
		$z=(1-\lambda)u+\lambda y$ satisfying $\mathcal{T}z^4 =0.$
	\end{center} This implies $z=(1-\lambda)u+\lambda y=0$, and then for all $i$, $$(1-\lambda)u_i\ge0,  \lambda y_i\ge0\mbox{ and }(1-\lambda)u_i+\lambda y_i=0.$$ So, we must have $u=y=0$,
	a contradiction. Therefore, $\mathcal{T}$ is  strictly copositive.
\end{proof}

\begin{theorem}\label{thm:32}
Let $\mathcal{T} = (t_{ijkl})$ be a 4th-order 3-dimensional symmetric tensor. Suppose $$|t_{ijkl}| = t_{iiii}= t_{iijj}=1,t_{iiij}t_{ijjj}=-1\mbox{ for all }i,j,k,l \in \{1,2,3 \}, i \neq j, i \neq k, k \neq i.$$ Then $\mathcal{T}$ is strictlly copositive if and only if
there is  at least $1$ in $\{t_{1123},t_{1223},t_{1233}\}$ and  for $i\ne j, j\ne k, i\ne k$, $$t_{iijk}=-1, t_{iiij}+t_{iiik}\ge0.$$
\end{theorem}

\begin{proof}
\textbf{Necessity}.  For $x=(1,1,1)^\top$, we have
$$\aligned
	\mathcal{T} x^4 =& x_1^4 + x_2^4 + x_3^4 +6x_1^2 x_2^2 + 6x_1^2 x_3^2 + 6x_2^2 x_3^2 +4t_{1112}x_1^3 x_2+ 4t_{1113} x_1^3 x_3 \\
		&+ 4t_{1222}x_1 x_2^3 +4t_{2223}x_2^3 x_3 + 4t_{1333}x_1 x_3^3 + 4t_{2333} x_2 x_3^3\\
		&+12t_{1123}x_1^2 x_2 x_3 + 12 t_{1223}x_1 x_2^2 x_3 + 12t_{1233} x_1 x_2 x_3^2 \\
	=&21+12(t_{1123}+t_{1223}+t_{1233}) > 0,
\endaligned$$
and hence, $$t_{1123}+t_{1223}+t_{1233}>-\dfrac{21}{12}.$$
Since $|t_{ijkl}|=1$, then $t_{1123}=t_{1223}=t_{1233}\ne-1$, and so,  the condition that there is at least one $1$ in $\{t_{1123},t_{1223},t_{1233}\}$  is necessary.

Now we show the necessity of  the other condition that for $i\ne j, j\ne k, i\ne k$, $t_{iijk}=-1$ and $t_{iiij}+t_{iiik}\ge0$.  Let $t_{1123}=-1$ without the generality. Then  $2\ge t_{1223}+t_{1233}\ge0$ by the condition that there is at least one $1$ in $\{t_{1123},t_{1223},t_{1233}\}$.

Assume the inequality that $t_{1112}+t_{1113}\ge0$ doesn't hold.  Then $t_{1112}=t_{1113}=-1$,  and moreover, $t_{1222}=t_{1333}=1$ by the codition $t_{iiij}t_{ijjj}=-1$. By this time,  for $x=(3,1,1)^\top$, noticing $t_{2223}t_{2333}=-1 \Rightarrow t_{2223}+t_{2333}=0$, we have
$$ \aligned
\mathcal{T} x^4 =& x_1^4 + x_2^4 + x_3^4 + 6x_1^2 x_2^2 + 6x_1^2 x_3^2 + 6x_2^2 x_3^3 - 12x_1^2 x_2 x_3 + 12 t_{1223}x_1 x_2^2 x_3 + 12t_{1233} x_1 x_2 x_3^2\\
&-4 x_1^3 x_2 - 4 x_1^3 x_3 + 4 x_1 x_2^3 + 4x_1 x_3^3 + 4t_{2223}x_2^3 x_3 +4t_{2333}x_2 x_3^3\\
=&83+54+54+6-108+36(t_{1223}+t_{1233})-108-108+12+12+4(t_{2223}+t_{2333})\\
\leq& 89+36\times2-192=-31<0,
\endaligned$$
which contradicts to the strict copositivity of $\mathcal{T}$. So, we must have $t_{1112}+t_{1113}\ge0$.

\textbf{Sufficiency}.  From Lemma \ref{lem:26} and the condition that $  t_{iiii}= t_{iijj}=1,t_{iiij}t_{ijjj}=-1$  for all $ i,j,k,l \in \{1,2,3 \}, i \neq j, i \neq k, k \neq i,$  it follows that each $2$-dimensional principal subtensor is strictly copositive,  and so,  there exists $$y\in\mathbb{R}^3_+\setminus\{0\}\mbox{ such that }\mathcal{T}y^4>0.$$
By Theorem \ref{thm:31}, we only show that $$x\in\mathbb{R}_+^3\mbox{  and  }	\mathcal{T}x^4=0\ \Longrightarrow\ x=0.$$

\textbf{Case 1}. $t_{1123}=t_{1223}=t_{1233}=1$.  Let $t_{1222}=-t_{1112}=t_{1333}=-t_{1113}=t_{2223}=-t_{2333}=1$ without the generality.
Then $\mathcal{T} x^4$ may be rewritten as follow, $$\mathcal{T} x^4 =(x_1+x_2+x_3)^4-8(x_1^3x_2+x_1^3x_3+x_2x_3^3).$$

\textbf{Case 2}.  There is only two $1$ in  $\{t_{1123},t_{1223},t_{1233}\}$. We might take $t_{1123}=-1, t_{1223}=t_{1233}=1$ and $t_{2223}=-t_{2333}=1$. Obviously, the condition that $t_{1112}+t_{1113}\ge0$ is equivalent to $$t_{1112}=t_{1113}=1\mbox{ or }t_{1112}t_{1113}=-1.$$
Then $\mathcal{T} x^4$ may be rewritten as
$$\mathcal{T} x^4 =(x_1+x_2+x_3)^4-8(x_1x_2^3+x_1x_3^3+x_2x_3^3)-24x_1^2x_2x_3,$$
or $$\mathcal{T} x^4 =(x_1+x_2+x_3)^4-8(x_1x_2^3+x_1^3x_3+x_2x_3^3)-24x_1^2x_2x_3,$$
or $$\mathcal{T} x^4 =(x_1+x_2+x_3)^4-8(x_1^3x_2+x_1x_3^3+x_2x_3^3)-24x_1^2x_2x_3.$$

\textbf{Case 3}.  There is only one $1$ in  $\{t_{1123},t_{1223},t_{1233}\}$. We might take $t_{1123}=t_{1223}=-1, t_{1223}=1$. Obviously, the conditions that $t_{1112}+t_{1113}\ge0$ and $t_{1222}+t_{2223}\ge0$ are equivalent to $$t_{1112}=t_{1113}=1\mbox{ or }t_{1112}t_{1113}=-1$$ and $$t_{1222}=t_{2223}=1\mbox{ or }t_{1222}t_{2223}=-1.$$ That is,$$t_{1112}=t_{1113}=-t_{1222}=t_{2223}=1\mbox{ or }t_{1222}=t_{2223}=-t_{1112}=t_{1113}=1, \mbox{ or }$$
$$t_{1112}=-t_{1113}=-t_{1222}=t_{2223}=1\mbox{ or }-t_{1112}=t_{1113}=t_{1222}=-t_{2223}=1.$$
 Then $\mathcal{T} x^4$ may be rewritten as
$$\mathcal{T} x^4 =(x_1+x_2+x_3)^4-8(x_1x_2^3+x_1x_3^3+x_2x_3^3)-24x_1x_2x_3(x_1+x_2),$$
or $$\mathcal{T} x^4 =(x_1+x_2+x_3)^4-8(x_1^3x_2+x_1x_3^3+x_2x_3^3)-24x_1x_2x_3(x_1+x_2),$$
or $$\mathcal{T} x^4 =(x_1+x_2+x_3)^4-8(x_1x_2^3+x_1^3x_3+x_2x_3^3)-24x_1x_2x_3(x_1+x_2),$$
or $$\mathcal{T} x^4 =(x_1+x_2+x_3)^4-8(x_1^3x_2+x_1x_3^3+x_2^3x_3)-24x_1x_2x_3(x_1+x_2).$$

It is easy to verify that for   the above all expressions $\mathcal{T} x^4$,  the   equation  $\mathcal{T} x^4 =0$ has only one real root $x_1=x_2=x_3=0$ in non-negativet octant $\mathbb{R}_+^3$. By Theorem \ref{thm:31}, $\mathcal{T} $ is strictly copositve.  This completes the proof.
\end{proof}

\begin{theorem}\label{thm:33}
Let $\mathcal{T} = (t_{ijkl})$ be a 4th-order 3-dimensional symmetric tensor with its entries $$t_{iiii}=t_{iiij}=-t_{iijj}=1,  t_{iijk}\in \{-1,0,1\},i,j,k=1,2,3,i\neq j,i \neq k, j \neq k.$$ Then $\mathcal{T}$ is strictly copositive if and only if $$t_{1123}+t_{1223}+t_{1233}\ge0.$$
\end{theorem}
\begin{proof}
	\textbf{Necessity}.  For $x=(1,1,1)^\top$,  we have
$$\aligned
\mathcal{T} x^4 =& x_1^4 + x_2^4 + x_3^4 -6x_1^2 x_2^2 - 6x_1^2 x_3^2 -6x_2^2 x_3^2  \\
&+4x_1^3 x_2+ 4 x_1^3 x_3+ 4x_1 x_2^3 +4x_2^3 x_3 + 4x_1 x_3^3 + 4 x_2 x_3^3\\
&+12t_{1123}x_1^2 x_2 x_3 + 12 t_{1223}x_1 x_2^2 x_3 + 12t_{1233} x_1 x_2 x_3^2 \\
=&9 +12(t_{1123}+   t_{1223}+  t_{1233} )>0.
\endaligned$$
That is, $t_{1123}+ t_{1223}+  t_{1233} >-\dfrac34,$
and hence, $$t_{1123}+t_{1223}+t_{1233}\ge0$$
since $t_{iijk}\in \{-1,0,1\}.$

\textbf{Sufficiency}. It follows from the condition that $t_{iijk}\in \{-1,0,1\}$ that
$$t_{1123}+t_{1223}+t_{1233}\ge0\Longleftrightarrow\begin{cases}
t_{1123}\in\{0,1\}, t_{1223}\in\{0,1\}, t_{1233}\in\{0,1\};\\
t_{1123}=-1, \begin{cases}
	t_{1223}\in\{0,1\}, t_{1233}=1;\\
	t_{1223}=1, t_{1233}\in\{0,1\};
\end{cases}\\
t_{1123}=0, t_{1223}t_{1233}=-1;\\
t_{1123}=1,\begin{cases}
	t_{1223}t_{1233}=0;\\
	t_{1223}t_{1233}=-1.
\end{cases}\\
\end{cases}$$
So, there is at most one $-1$ in $\{t_{1123}, t_{1223}, t_{1233}\}$ and both $1$ and $-1$ always come in a pair.

\textbf{Case 1}.  There is  actually one $-1$ in $\{t_{1123}, t_{1223}, t_{1233}\}$.  Let $t_{1123}=-1$, $t_{1223}=1, t_{1233}\in\{0,1\}$ without the generality. Then  $\mathcal{T} x^4 $ may be expressed as follows,
$$\aligned
\mathcal{T} x^4= &x_1^4 + x_2^4 + x_3^4 + 4x_1^3 x_2 + 4x_1^3 x_2 + 4x_1 x_2^3 + 4x_2^3 x_3 + 4x_1 x_3^3 + 4x_2 x_3^3\\
&-6x_1^2 x_2^2 -6x_1^2 x_3^2 -6x_2^2 x_3^2 -12  x_1^2 x_2 x_3 +12 x_1 x_2^2 x_3 + 12t_{1233} x_1 x_2 x_3^2\\
\ge &x_1^4 + x_2^4 + x_3^4 + 4x_1^3 x_2 + 4x_1^3 x_2 + 4x_1 x_2^3 + 4x_2^3 x_3 + 4x_1 x_3^3 + 4x_2 x_3^3\\
&-6x_1^2 x_2^2 -6x_1^2 x_3^2 -6x_2^2 x_3^2 -12 x_1^2 x_2 x_3+12 x_1 x_2^2 x_3\\
=&(x_1 + x_2 + x_3)^4 - 12(x_1^2 x_2^2 +x_1^2 x_3^2 +x_2^2 x_3^2)- 12x_1 x_2 x_3( 2 x_1 +x_3) .
\endaligned$$
Let $$\mathcal{T}' x^4=(x_1 + x_2 + x_3)^4 - 12(x_1^2 x_2^2 +x_1^2 x_3^2 +x_2^2 x_3^2)- 12x_1 x_2 x_3( 2 x_1 +x_3).$$ Then, solve the equation $\mathcal{T}' x^4=0$ in the non-negative orthant $\mathbb{R}_+^3$ to yield $x=0$.  Simultaneously, by Lemma \ref{lem:26},  the condition that $t_{iiii}= t_{iiij}=-t_{iijj}=1$ implies  the strict copositivity of each $2$-dimensional principal subtensor.  So an application of Theorem \ref{thm:31} erects the strict copositivity of $\mathcal{T}' $, and hence, $\mathcal{T}$ is strictly copositive.

\textbf{Case 2}.  There is not  $-1$ in $\{t_{1123}, t_{1223}, t_{1233}\}$.  Then $t_{1123}\ge0, t_{1223}\ge0, t_{1233}\ge0$, and moreover, $\mathcal{T} x^4 $ may be rewritten as follows,
$$\aligned
\mathcal{T} x^4= &x_1^4 + x_2^4 + x_3^4 + 4x_1^3 x_2 + 4x_1^3 x_2 + 4x_1 x_2^3 + 4x_2^3 x_3 + 4x_1 x_3^3 + 4x_2 x_3^3\\
&-6x_1^2 x_2^2 -6x_1^2 x_3^2 -6x_2^2 x_3^2 +12 t_{1123} x_1^2 x_2 x_3 +12t_{1223} x_1 x_2^2 x_3 + 12t_{1233} x_1 x_2 x_3^2\\
\ge &x_1^4 + x_2^4 + x_3^4 + 4x_1^3 x_2 + 4x_1^3 x_2 + 4x_1 x_2^3 + 4x_2^3 x_3 + 4x_1 x_3^3 + 4x_2 x_3^3\\
&-6x_1^2 x_2^2 -6x_1^2 x_3^2 -6x_2^2 x_3^2\\
=&(x_1 + x_2 + x_3)^4 - 12(x_1^2 x_2^2 +x_1^2 x_3^2 +x_2^2 x_3^2)- 12x_1 x_2 x_3(  x_1 +x_3+x_3).
\endaligned$$
Similarly, it is not difficult to verify that in  $R_+^3$, the unique solution the equation, $$\mathcal{T} ''x^4=(x_1 + x_2 + x_3)^4 - 12(x_1^2 x_2^2 +x_1^2 x_3^2 +x_2^2 x_3^2)- 12x_1 x_2 x_3(  x_1 +x_3+x_3)=0$$
 is $x=0.$ Therefore, $\mathcal{T}$ is strictly copositive by Theorem \ref{thm:31}.
\end{proof}

\begin{corollary}\label{cor:34}
	Let $\mathcal{T} = (t_{ijkl})$ be a 4th-order 3-dimensional symmetric tensor with its entries $$t_{iiii}=t_{iiij}=1, t_{iijj}, t_{iijk}\in \{-1,0,1\},i,j,k=1,2,3,i\neq j,i \neq k, j \neq k.$$ Then $\mathcal{T}$ is strictly copositive   if $t_{1123}+t_{1223}+t_{1233}\ge0.$
\end{corollary}

\begin{corollary}\label{cor:35}
Let $\mathcal{T} = (t_{ijkl})$ be a 4th-order 3-dimensional symmetric tensor. If $t_{iiii} \geq 1,t_{iiij} \geq 1,t_{iijj} \geq-1,t_{iijk} \geq0$ for all $i,j,k \in \{1,2,3\},i\neq j,i\neq k, j \neq k$, then $\mathcal{T}$ is strictly copositive.
\end{corollary}

\begin{theorem}\label{thm:36}
Let $\mathcal{T} = (t_{ijkl})$ be a 4th-order 3-dimensional symmetric tensor with its entries $$t_{iiii}=t_{iiij}=-t_{iijk}=1, t_{iijj}\in \{-1,0,1\}, i,j,k=1,2,3,i\neq j,i \neq k, j \neq k. $$ Then $\mathcal{T}$ is strictly copositive if and only if $t_{iijj}\in\{0,1\}, i,j=1,2,3,i\neq j$ and there is at least two 1 in $\{t_{1122},t_{1133},t_{2233}\}$.
\end{theorem}

\begin{proof}
	\textbf{Necessity}.  For $x=(1,1,1)^\top$,  we have
$$\aligned
\mathcal{T} x^4 =& x_1^4 + x_2^4 + x_3^4 +6t_{1122}x_1^2 x_2^2 + 6t_{1133}x_1^2 x_3^2 + 6t_{2233}x_2^2 x_3^2  \\
&+4x_1^3 x_2+ 4 x_1^3 x_3+ 4x_1 x_2^3 +4x_2^3 x_3 + 4x_1 x_3^3 + 4 x_2 x_3^3\\
&-12x_1^2 x_2 x_3 -12x_1 x_2^2 x_3 - 12 x_1 x_2 x_3^2 \\
=&6(t_{1122}+ t_{1133} + t_{2233})-9>0.
\endaligned$$
That is, $t_{1122}+t_{1133}+t_{2233}> \dfrac32$, which is equivalent to $t_{iijj}\in\{0,1\}, i,j=1,2,3,i\neq j$ and there is at least two 1 in $\{t_{1122},t_{1133},t_{2233}\}$.

\textbf{Sufficiency}. Without loss the generality, let $t_{1122}=t_{1133}=1, t_{2233}\in\{0,1\}$.
$$\aligned
\mathcal{T} x^4 &\geq x_1^4 + x_2^4 + x_3^4 + 4x_1^3 x_2 + 4x_1^3 x_2 + 4x_1 x_2^3 + 4x_2^3 x_3 + 4x_1 x_3^3 + 4x_2 x_3^3\\
&~~+6x_1^2 x_2^2 + 6 x_1^2 x_3^2 + 0 \cdot x_2^2 x_3^2 -12 x_1^2 x_2 x_3 - 12 x_1 x_2^2 x_3 - 12 x_1 x_2 x_3^2\\
&= (x_1+x_2+x_3)^4 - 6x_2^2 x_3^2-24x_1 x_2 x_3 ( x_1 +x_2+ x_3 ).
\endaligned$$
Using the similar proof technique of Theorem  \ref{thm:33},  solve the equation  $$\hat{\mathcal{T}} x^4=(x_1+x_2+x_3)^4 - 6x_2^2 x_3^2- 24x_1 x_2 x_3 ( x_1 +x_2+ x_3 )=0$$ to yield $x=0$ in $\mathbb{R}^3_+$.
So, $\mathcal{T}$ is strictly copositive.
\end{proof}

\end{document}